\documentclass[a4paper,11pt]{article}

\usepackage{amssymb}
\usepackage{latexsym}
\usepackage{amsfonts}
\usepackage{amsmath}
\usepackage{amsthm}
\usepackage{hyperref}
\usepackage{tikz}
\usepackage[left=1in,right=1in,top=1in,bottom=1in]{geometry}
\usepackage{parskip}
\usepackage[scale=0.95]{tgheros}
\usepackage[T1]{fontenc}

\theoremstyle{plain}
\newtheorem{lemma}{Lemma}
\newtheorem{corollary}[lemma]{Corollary}
\newtheorem{theorem}[lemma]{Theorem}
\newtheorem{proposition}[lemma]{Proposition}

\theoremstyle{definition}
\newtheorem{definition}[lemma]{Definition}

\numberwithin{equation}{section}
\numberwithin{lemma}{section}

\DeclareMathOperator{\Aut}{Aut}

\DeclareMathOperator{\AGL}{AGL}

\DeclareMathOperator{\Rep}{Rep}

\DeclareMathOperator{\Sym}{Sym}

\DeclareMathOperator{\diff}{diff}
\DeclareMathOperator{\PermAut}{PermAut}

\definecolor{blue}{rgb}{0,0.4,0.6}
\definecolor{ccqqqq}{rgb}{0.8,0,0}
\definecolor{qqttcc}{rgb}{0,0.2,0.8}
\definecolor{qqwwzz}{rgb}{0,0.4,0.6}
\definecolor{shadecolor}{rgb}{0,0.4,0.6}
\definecolor{qqqqff}{rgb}{0,0.4,0.6}

\renewcommand{\b}{\underline}

\newcommand{\F}{\mathbb F}
\newcommand{\D}{\mathcal D}
\newcommand{\PG}{\mathcal{PG}}
\newcommand{\EG}{\mathcal{EG}}

\newcommand{\G}{\mathcal{G}_{23}}
\renewcommand{\P}{\mathcal P}

\newcommand{\RM}{\mathcal{RM}}

\begin{document}      % Start the document.

\title{\texorpdfstring{$s$}{s}-Elusive Codes in Hamming Graphs
}

\author{
 Daniel R. Hawtin
 \thanks{
  This research was partially supported by an Australian Postgraduate Award and a University of Western Australia Safety-Net Top-Up Scholarship. This work has been supported in part by the Croatian Science Foundation under project 6732.
 }
 \thanks{
  The author thanks Neil Gillespie, Cheryl Praeger and Andrea {\v S}vob for kindly reading drafts of this manuscript, and two anonymous referees for their helpful comments on a previous version.
 }
}

\date{
 \small{
  \emph{
   Department of Mathematics, University of Rijeka\\
   Rijeka, Croatia, 51000\\
   \href{mailto:dan.hawtin@gmail.com}{dan.hawtin@gmail.com}\\
   \vspace{0.25cm}
  }
  \today
 }
}

\maketitle

\begin{abstract}
A \emph{code} is a subset of the vertex set of a \emph{Hamming graph}. The \emph{set of $s$-neighbours} of a code is the set of all vertices at Hamming distance $s$ from their nearest codeword. A code $C$ is \emph{$s$-elusive} if there exists a distinct code $C'$ that is equivalent to $C$ under the full automorphism group of the Hamming graph such that $C$ and $C'$ have the same set of $s$-neighbours.

We show that the minimum distance of an $s$-elusive code is at most $2s+2$, and that an $s$-elusive code with minimum distance at least $2s+1$ gives rise to a $q$-ary $t$-design with certain parameters. This leads to the construction of: an infinite family of $1$-elusive and completely transitive codes, an infinite family of $2$-elusive codes, and a single example of a $3$-elusive code. Answers to several open questions on elusive codes are also provided.
\end{abstract}

\section{Introduction}\label{introd}

A \emph{code} in a Hamming graph $\varGamma=H(m,q)$ is a subset $C$ of its vertex set $V\varGamma$. The elements of $C$ are called \emph{codewords} and the automorphism group of $C$ is the setwise stabiliser of $C$ in the full automorphism group of $H(m,q)$. An $s$-neighbour of $C$ is a vertex $\alpha$ whose nearest codeword in $C$ is Hamming distance $s$ from $\alpha$. A code $C$ is called \emph{$s$-elusive} if there exists an equivalent code $C'$ to $C$ such that the sets of $s$-neighbours of $C$ and $C'$ are the same. Note that the notion of equivalence used here is more general than the standard one; see Section~\ref{notation}.

The concept studied here is a generalisation of one originally studied in \cite{ntrcodes}. We consider the question of whether, given a code $C$ in a Hamming graph $H(m,q)$, the automorphism group $\Aut(C_s)$ of the set $C_s$ of $s$-neighbours could be larger than the automorphism group $\Aut(C)$ of the code itself (see Section~\ref{notation}). This question was encountered, for $s=1$, when Gillespie and Praeger were deciding upon the definition for a neighbour-transitive code (see \cite{neilphd}). In \cite{ntrcodes} they give an affirmative answer via the construction of an infinite family of examples. Similarly, the significance of the existence of $s$-elusive codes relates to the precise definition of $s$-neighbour-transitive codes (see \cite{ef2nt,aas2nt,minimal2nt,ondimblock}).

Theorem~\ref{maintheorem} exhibits examples of $s$-elusive codes, for $s=1,2$ and $3$. The definition of the relevant Reed-Muller codes is given at the beginning of Section~\ref{reedmuller}, and can be found for instance in \cite[Section 5.4]{assmus1994designs}; the definition of the Preparata codes can be found in \cite[(16.12)]{cameron1991designs}. Part 1 of Theorem~\ref{maintheorem} is proved in Section~\ref{reedmuller} and the remaining parts in Section~\ref{selusive}. Note that a code $C$ is \emph{$G$-completely transitive}, for a group $G\leq \Aut(C)$, if each $C_i$ is a $G$-orbit, for $ i\in \{0,\ldots, \rho\}$, where $\rho$ is the covering radius of $C$ (see, for instance, \cite{giudici1999completely}), and we say that a linear code $C$ of length $m$, dimension $\ell$ and minimum distance $\delta$ has \emph{parameters} $[m,\ell,\delta]$. Note also that each code appearing in Theorem~\ref{maintheorem} is known to be completely regular with covering radius $2s$ (see \cite{borges2019completely}), and that the result that $\RM_q(k,d)$ is completely transitive is new, to the best of the authors knowledge.

\begin{theorem}\label{maintheorem}
 \begin{enumerate}
  \item The Reed-Muller codes $\RM_q(k,d)$, where $q$ is a prime power and $k=(q-1)d-2$, are completely transitive and $1$-elusive with parameters $[q^d,q^d-d-1,\delta]$ where $\delta=4$ when $q=2$ and $\delta=3$ otherwise.
  \item The Preparata codes $\P(2d)$ in $H(2^{2d},2)$ (which are non-linear) are $2$-elusive with minimum distance $\delta=6$ and size $2^{2^{2d}-4d}$.
  \item The punctured code of the even weight subcode of the perfect binary Golay code is $3$-elusive with parameters $[22,11,7]$.
 \end{enumerate}
\end{theorem}

For a code $C$ to be $s$-elusive, there must be an automorphism $x\in \Aut(C_s)\setminus \Aut(C)$. It follows that $C^x$ and $C$ are not equal, but are equivalent codes, each with the same $s$-neighbour set $C_s$. As such, given knowledge only of the $s$-neighbour set and minimum distance of an $s$-elusive code, knowledge of the code itself remains elusive. Whether such codes exist seems to be related to the \emph{minimum distance} $\delta$ of the code, namely the smallest distance between two distinct codewords.  In \cite{ntrcodes} it is shown that (i) if $C$ is a $1$-elusive code then it has minimum distance $\delta\leq 4$, (ii) that if $\delta=4$ then $q=2$, and (iii) an infinite family of binary $1$-elusive codes with $\delta=4$ exists.

Requiring $\delta\geq 2s+1$ in what follows avoids certain trivial cases and technicalities, making some interesting results possible. In particular, Theorems~\ref{design} and~\ref{mindistupbound} together generalise \cite[Theorem~1]{ntrcodes}, showing that the minimum distance of an $s$-elusive code is at most $2s+2$, and that any $s$-elusive code with minimum distance at least $2s+1$ has a set of $q$-ary $s$-$(m,2s,1)$ designs associated to it. Note that the latter fact allowed for the identification of those codes in Parts 2 and 3 of Theorem~\ref{maintheorem}. Designs often arise as subsets of codes. For instance, \cite[Theorem~2.12]{borges2019completely} states that the set of all weight $k$ vertices of a completely regular code having minimum distance $\delta$ in $H(m,q)$ form a $q$-ary $\left\lfloor \frac{\delta}{2}\right\rfloor$-$(m,k,\lambda_k)$ design, for some ineteger $\lambda_k$.

In \cite{elusive}, for each $q\geq 3$, an infinite family of $1$-elusive codes with $\delta=3$ in $H(m,q)$ was constrcuted. It was observed in that paper that for all known examples the length $m$ of the code is divisible by the alphabet size $q$. In \cite[Question 1.3]{elusive} it was asked whether this was true in general. This holds in the binary case, by \cite[Theorem 1]{ntrcodes}, since this implies that $m(q-1)=m$ must be even, regardless of $\delta$. The author thanks Andries Brouwer for sending in private correspondence \cite{brouwersElusive} the basis of the beautiful argument contained in Section~\ref{divis}. This argument shows that the answer to the question is `yes', that is, for an $s$-elusive code to exist in $H(m,q)$ it must be that $q$ divides $m$. This generalises and simplifies \cite[Theorem~1.2]{structure} in the unpublished manuscript of the author.

The family $\RM_q(k,d)$ of $1$-elusive codes, as in Part 1 of Theorem~\ref{maintheorem}, provides answers to further questions raised in \cite{elusive}.
\begin{enumerate}
 \item In that paper there are only two images of each example code $C$ under $\Aut(C_1)$; \cite[Question 1.4]{elusive} asks if this is always the case.
 \item A code $C$ is $G$-\emph{neighbour-transitive} if each of the sets $C$ and $C_1$ are $G$-orbits for some group $G$. In \cite[Question 1.5]{elusive} it is asked whether the images under $\Aut(C_1)$ of a $1$-elusive code $C$ which is $\Aut(C)$-neighbour-transitive must be pairwise disjoint.
\end{enumerate}

\begin{theorem}\label{answerforques}
 Let $C=\RM_q(k,d)$, as in Part 1 of Theorem~\ref{maintheorem}. If $q$ is a power of the prime $p$ then:
 \begin{enumerate}
  \item there are at least $p$ distinct images of $C$ under $\Aut(C_1)$; and,
  \item there exists some $x\in \Aut(C_1)\setminus \Aut(C)$ such that $\b 0 \in C\cap C^x$.
 \end{enumerate}
\end{theorem}

It is of note that studying the $s$-neighbour set of a code, usually when $s$ is equal the covering radius $\rho$, arises in cryptography. \emph{Bent functions} are functions with ``maximal non-linearity'', which turns out to be the same as being a vertex in $H(q^d,q)$ at distance $\rho$ from the first order Reed-Muller code $\RM_q(1,d)$; see \cite[Chapter~14, Section~5]{macwilliams1978theory}, or \cite{carlet2006hyper,tokareva2015bent} for extensions of this concept.

The next section introduces some notation, Section~\ref{divis} answers \cite[Question 1.3]{elusive}, before Sections~\ref{reedmuller} and~\ref{selusive} provide the proof of Theorem~\ref{maintheorem}.

\section{Preliminaries}\label{notation}

Let the two sets $M$ and $Q$ have sizes $m$ and $q$ respectively. For any set $S$ with $0\in S$ write $S^\times=S\setminus \{0\}$. The vertex set of the Hamming graph $\varGamma=H(m,q)$ consists of all $m$-tuples with entries labelled by the set $M$ and taken from the set $Q$. An edge exists between two vertices if they differ as $m$-tuples in exactly one position. For vertices $\alpha,\beta\in\varGamma$ the \textit{Hamming distance} $d(\alpha,\beta)$ (that is the distance in $\varGamma$) is the number of entries in which $\alpha$ and $\beta$ differ.

For any vertex $\alpha\in\varGamma$, the \emph{set of $r$-neighbours} of $\alpha$ is $\varGamma_r(\alpha)=\{\beta\in\varGamma \mid d(\alpha,\beta)=r\}$. The set of entries in which $\alpha,\beta\in\varGamma$ differ is $\diff(\alpha,\beta)=\{i\in M\mid \alpha_i\neq\beta_i\}$.

Let $C$ be a code in $H(m,q)$. Then the \emph{minimum distance} of $C$ is $\delta=\min\{d(\alpha,\beta)\mid \alpha,\beta\in C,\alpha\neq \beta\}$. For a vertex $\alpha$ of $\varGamma$, define $d(\alpha,C)=\min\{d(\alpha,\beta) \mid \beta\in C\}$. Then the \textit{covering radius} $\rho =\max\{d(\alpha,C)\mid\alpha\in\varGamma\}$. As in Section~\ref{introd}, for any $r\leq \rho$ let $C_r=\{\alpha\in\varGamma \mid d(\alpha,C)=r\}$. Note that if $\delta\geq 2r$, then the set of $r$-neighbours $C_r$ of the code $C$ satisfies $C_r=\cup_{\alpha\in C}\varGamma_r(\alpha)$ and if $\delta\geq 2r+1$ this is a disjoint union.

The \emph{repetition code} $\Rep(m,q)$ in $H(m,q)$ is the code consisting of all $m$-tuples $(a,\ldots,a)$ where $a\in Q$. A code $C$ is \emph{linear} if $Q\cong \F_q$ and $C$ is a subspace of the vertex set $V\varGamma\cong\F_q^m$. If $C$ is a linear code then $\Aut(C)$ contains the subgroup $T_C$ consisting of all translations $t_\alpha$, where $\alpha\in C$, defined by $\beta\mapsto \alpha+\beta$ for all $\beta\in V\varGamma$. We denote the dual of a linear code $C$ under the standard inner product by $C^\perp$. The code $\Rep(m,2)$ in $H(m,2)$ is linear and its dual $\Rep(m,2)^\perp$ is the code consisting of all vertices of even weight. The \emph{even-weight subcode} of any code $C$ in $H(m,2)$ is given by $C\cap \Rep(m,2)^\perp$.

Let $S_n$ denote the symmetric group on $\{1,\ldots,n\}$. The automorphism group $\Aut(\varGamma)$ of the Hamming graph is the semi-direct product $B\rtimes L$, where $B\cong S_q^m$ and $L\cong S_m$ (see \cite[Theorem 9.2.1]{brouwer}). Let $g=(g_1,\dots,g_m)\in B$, $\sigma\in L$ and $\alpha=(\alpha_1,\ldots,\alpha_m)\in\varGamma$. Then $g$ and $\sigma$ act on $\alpha\in \varGamma$ as follows:
\begin{equation*}
\alpha^g =(\alpha_1^{g_1},\ldots,\alpha_m^{g_m})\quad\text{and}\quad
\alpha^\sigma=(\alpha_{1{\sigma^{-1}}},\ldots,\alpha_{m{\sigma^{-1}}}).
\end{equation*}

The automorphism group of a code $C$ in $\varGamma=H(m,q)$ is $\Aut(C)=\Aut(\varGamma)_C$, the setwise stabiliser of $C$ in $\Aut(\varGamma)$. The group of \emph{pure permutations} on entries is $\PermAut(C)=\Aut(C)\cap L$. This notation will be used for any subset of vertices, in particular the automorphism group of the set of $r$-neighbours of $C$ is $\Aut(C_r)=\Aut(\varGamma)_{C_r}$. 

Two codes, $C$ and $C'$, in $H(m,q)$, are \textit{equivalent} if there exists $x\in \Aut(\varGamma)$ such that $C^x=C'$. Equivalence preserves minimum distance. (See \cite[Lemma 4]{ntrcodes}).

\section{Alphabet Size Divides Length}\label{divis}

The \emph{adjacency matrix} of a graph has rows and columns indexed by the vertices of the graph, with an entry $1\in\mathbb{R}$ if the corresponding vertices are adjacent and $0\in\mathbb{R}$ otherwise. Let $A$ be the adjacency matrix of the Hamming graph. A subset of the vertex set of a graph, and hence a code $C$, can be represented by a \emph{characteristic vector} $\chi(C)$, where the entries are labelled by the vertices of the graph and take the value $1\in\mathbb{R}$ if the vertex is in $C$ and $0\in\mathbb{R}$ otherwise. It follows that $A\cdot\chi(C)$ is related to the characteristic vector of $C_1$, the entry of $A\cdot\chi(C)$ corresponding to the vertex $\beta$ takes the value $|\varGamma_1(\beta)\cap C|$. In particular, if $\delta\geq 3$ then each element of $C_1$ is distance $1$ from a unique codeword, and hence $A\cdot\chi(C)=\chi(C_1)$. To generalise this, note that the value of $A^s$ in the $i$-th column and $j$-th row gives the number of paths of length $s$ between the vertices $i$ and $j$. Since two vertices at distance $s$ differ in precisely $s$ positions, there are $s!$ paths of length $s$ between them. Also, if $\delta\geq 2s+1$ then each element of $C_s$ is distance $s$ from a unique codeword. Hence, $A^s\cdot \chi(C)=s!\chi(C_s)$. Note that, in general, $K_s(A)\cdot \chi(C)=\chi(C_s)$, where $K_s$ is a \emph{Krawtchouk polynomial}, but here the condition $\delta\geq 2s+1$ allows this expression to be simplified.

%
% \begin{lemma}\label{twocodes}
%  Let $s=\{1,\ldots,\rho\}$ and suppose there exist codes $C$ and $C'$ in $H(m,q)$ such that $C_1=C'_1$, both with minimum distance $\delta\geq 2s+1$. Then $q$ divides $m$.
% \end{lemma}
%
% \begin{proof}
%  Let $A$ be the adjacency matrix of the Hamming graph $H(m,q)$ and $u$ and $v$ be the characteristic vectors of $C$ and $C'$ respectively. Then, since $\delta\geq 3$, $Au$ is the charcteristic vector of $C_1$. Moreover, since both codes satisfy $\delta\geq 3$, we have $Au=Av$. Since $u\neq v$, $A$ is non-singular and has at least one zero eigenvalue. The Hamming graph is the Cartesian product of $m$ copies of the complete graph on $q$ vertices $K_q$. Hence, by \cite[Theorem~2.3.4]{cvetkovic2008eigenspaces}, and the fact that the eigenvalues of $K_q$ are $-1$ and $q-1$, we see that the Hamming graph has eigenvalues $(m-i)(q-1)-i=(q-1)m-iq$, where $0\leq i\leq m$. Thus $A$ has zero eigenvalue when $i=(q-1)m/q$, which occurs only when $q\mid m$.
% \end{proof}

\begin{proposition}\label{twocodes}
 Let $s\in\{1,\ldots,\rho\}$ and suppose that there exist distinct codes $C$ and $C'$ in $H(m,q)$ such that $C_s=C'_s$, with both $C$ and $C'$ having minimum distance at least $2s+1$. Then $q$ divides $m$.
\end{proposition}

\begin{proof}(Basis of this argument comes from \cite{brouwersElusive})
 Let $A$ be the adjacency matrix of the Hamming graph $H(m,q)$ and let $u=\chi(C)$, $v=\chi(C')$. Since both $C$ and $C'$ have minimum distance at least $2s+1$, it follows (from the discussion immediately preceding this result) that $A^s u=s!\chi(C_s)=s!\chi(C'_s)=A^s v$. Since $u\neq v$, it follows that $A^s$, and hence also $A$, is singular and has at least one zero eigenvalue. The Hamming graph is the Cartesian product of $m$ copies of the complete graph $K_q$ on $q$ vertices. Thus, by \cite[Theorem~2.3.4]{cvetkovic2008eigenspaces} and the fact that the eigenvalues of $K_q$ are $-1$ and $q-1$, the Hamming graph has eigenvalues $(m-i)(q-1)-i=(q-1)m-iq$, where $0\leq i\leq m$. Since $A$ has an eigenvalue zero this implies $(q-1)m-iq=0$, for some integer $i$, and hence $q\mid m$.
\end{proof}

\begin{corollary}\label{elusiveqdivm}
 Let $C$ be an $s$-elusive code in $H(m,q)$ with $\delta\geq 2s+1$. Then $q$ divides $m$.
\end{corollary}

\begin{proof}
 If $C$ is an $s$-elusive code, then there exists $x\in \Aut(C_s)\setminus \Aut(C)$ such that $C^x\ne C$ but $C_s^x=C_s$. Hence, since $\delta\geq 2s+1$, Lemma~\ref{twocodes} applies with $C'=C^x$.
\end{proof}

\section{Elusive Reed-Muller Codes}\label{reedmuller}

This section concerns Part 1 of Theorem~\ref{maintheorem}, that is, we give an infinite family of $1$-elusive and completely transitive codes. Each code is the dual of a first order $q$-ary Reed-Muller code and is contained in the dual of the repetition code of the respective length.

Fix the following notation throughout this section. Let $q$ be a prime power, $Q=\F_q$ and $M=\F_q^d$, so that $V\varGamma$ is an $\F_q$-vector space. For $\alpha\in V\varGamma$, consider the following equations:
\begin{align}
 \sum_{v\in M}\alpha_v &= 0,\quad\text{and}, \label{alphavs}\\
 \sum_{v\in M} \alpha_v v&= 0.\label{valphavs}
\end{align}
Moreover, fix $k=(q-1)d-2$, as well as:
\[
 C=\RM_q(k,d)) \quad\text{and}\quad C'=\RM_q(k+1,d)=\Rep(q^d,q)^\perp,
\]
in $H(q^d,q)$ (where $\Rep(q^d,q)^\perp$ is the dual of the repetition code). The significance of (\ref{alphavs}) and (\ref{valphavs}) is that $\alpha\in C'$ if and only if $\alpha$ satisfies (\ref{alphavs}), and $\alpha\in C$ if and only if $\alpha$ satisfies both equations (\ref{alphavs}) and (\ref{valphavs}) (see \cite[Section 5.4]{assmus1994designs}).

The next lemma states some well-known facts about $C'$, the dual of the repetition code; see, for instance, \cite{macwilliams1978theory}.

\begin{lemma}\label{hatc}
 The code $C'$ is linear with dimension $q^d-1$, minimum distance $\delta'=2$, covering radius $\rho'=1$ and $|{C'}_1|=(q-1)q^{q^d-1}$.
\end{lemma}

The next result is also well known.

\begin{lemma}\label{reedmullerparams}\cite[Corollary 5.5.4 and Theorem 5.4.1]{assmus1994designs}
 The code $C$ has covering radius $\rho=2$, dimension $q^d-(d+1)$, and minimum distance \[\delta=\left\{
 \begin{array}{ll}
 4 & \text{if}\quad q=2,d\geq 2,\\
 3 & \text{if}\quad q\geq 3,d\geq 1.
 \end{array}\right.\]
\end{lemma}

\begin{lemma}\label{neighbourset}
 The sets $C_1$ and ${C'}_1$ of neighbours of $C$ and $C'$ satisfy $C_1={C'}_1$.
\end{lemma}

\begin{proof}
 Now, by Lemma~\ref{reedmullerparams}, $|C|=q^{q^d-(d+1)}$. Since $\delta'=2$ and $C\subset C'$ it follows that ${C}_1\subseteq {C'}_1$. Also, since $\delta\geq 3$, $|C_1|=m(q-1)|C|=q^d(q-1)q^{q^d-(d+1)}=q^{q^d}-q^{q^d-1}$, and thus $C_1={C'}_1$ by Lemma~\ref{hatc}.
\end{proof}

\begin{lemma}\label{rmelusive}
 The Reed-Muller code $C=\RM_q(k,d)$ is a $1$-elusive code.
\end{lemma}

\begin{proof}
 Now $\Aut(C_1)=\Aut(C'))$ because, by Lemma~\ref{neighbourset}, ${C}_1={C'}_1$, and, by Lemma~\ref{hatc}, $V\varGamma=C'\cup {C'}_1$. Since $C'$ is linear, $\Aut(C_1)(=\Aut(C'))$ contains the translation $t_{\alpha}$ by the vertex $\alpha$ for each $\alpha\in C'$. If $\alpha\in C'\setminus C$ then $t_{\alpha}$ does not fix $C$ setwise, so $t_\alpha\notin \Aut(C)$, and hence the image $C^{t_\alpha}\neq C$, so $C$ is $1$-elusive.
\end{proof}

Recall from Section~\ref{notation} that $\PermAut(C)=\Aut(C)\cap L$ is the group of pure permutations on entries fixing the code $C$. By \cite[Theorem 5]{berger1993automorphism}, $\PermAut(C)\cong\AGL(d,q)$. Since $C'$ is the dual of the repetition code in $H(m,q)$, it follows that $\PermAut(\RM_q(k+1,d))\cong S_m$. The proof of Theorem~\ref{answerforques} is below, which provides answers to two open questions regarding elusive codes.

\begin{proof}[Proof of Theorem~\ref{answerforques}]
 If $p$ is the characteristic of the field $\F_q$, then any non-trivial translation in $\Aut(C_1)$ has order $p$. As in the proof of Lemma~\ref{rmelusive} there is a translation in $\Aut(C_1)\setminus \Aut(C)$, so there are at least $p$ distinct images of $C$ under elements of $\Aut(C_1)$. This proves part 1. Note also that $\sigma\in \Aut(C')$ for any $\sigma\in \Sym(M)$, where $\sigma$ acts by permuting entries. However, by \cite[Theorem 5]{berger1993automorphism}, $\sigma\in \PermAut(C)$ if and only if $\sigma\in \AGL(d,q)$. Thus if $\sigma\in \Sym(M)\setminus \AGL(d,q)$, then $C^\sigma\neq C$. However $\b 0\in C^\sigma\cap C$, proving part 2.
\end{proof}

\begin{lemma}\label{reedmullercomptrans}
 The Reed-Muller code $C=\RM_q(k,d)$ is $\Aut(C)$-completely transitive.
\end{lemma}

\begin{proof}
 Since $C$ is linear, $\Aut(C)$ is transitive on $C$. By Lemma~\ref{reedmullerparams}, $C$ has covering radius $2$, so it remains to prove that $\Aut(C)$ acts transitively on $C_1$ and $C_2$. Since $\delta\geq 3$, $\b 0\in C$ and $\Aut(C)$ is transitive on $C$, to prove that $\Aut(C)$ is transitive on $C_1$ it is sufficient to prove $\Aut(C)_{\b 0}$ is transitive on the set of weight one vertices. Let $\nu$ be the weight one vertex with $\nu_i=a\in Q^\times$ for a unique $i\in M$. By \cite[Theorem 5]{berger1993automorphism}, $\PermAut(\RM_q(k,d))\cong \AGL(d,q)$ acting $2$-transitively as pure permutations on entries. Since $C$ is linear $\Aut(C)$ also contains a subgroup isomorphic to the multiplicative group $\F_q^\times$ acting as scalar multiplication. Hence, multiplying by $a^{-1}$ and then applying a permutation of the entries $\sigma\in \Aut(C)$ which maps $i$ to $0\in M$, will map $\nu$ to the weight one vertex $\mu$ with $\mu_0=1$.

 We now prove $\Aut(C)$ is transitive on $\varGamma_2({\b 0})\cap C_2$, which will complete the proof. Recall $C'=\RM_q(k+1,d)$. Now $\varGamma_2(\b 0)\cap C_2$ consists of the weight two vertices $\nu$ with $\nu_i=a\in Q^\times$, $\nu_j=-a$ for distinct $i,j\in M$. To see this, first note that each such vertex $\nu$ satisfies the condition in (\ref{alphavs}), but not the conditions in (\ref{valphavs}) and so $\nu\in C\setminus C'$. By Lemma~\ref{hatc}, $C'$ has minimum distance $2$ and, by Lemma~\ref{neighbourset}, $C_1={C'}_1$, and thus $\nu\in C_2$. Next, suppose $\nu'$ is an arbitrary vertex in $\varGamma_2(\b 0)$, with ${\nu'}_i\neq 0$,${\nu'}_j\neq 0$, for some $i\neq j$. If $\nu_i\neq -\nu_j$ then $\nu\in C_1$ since, by (\ref{valphavs}), $C$ contains the weight three vertex $\alpha\in\varGamma_1(\nu)$ with $\alpha_i=\nu'_i$, $\alpha_j=\nu'_j$ and $\alpha_{i+j}=-{\nu'}_i-{\nu'}_j$. Hence $\nu$ has the form claimed. Finally, we can map $\nu\in\varGamma_2(\b 0)\cap C_2$ to the weight two vertex $\mu$, where $\mu_{0}=1$, $\mu_{e_1}=-1$, by multiplying by $a^{-1}$ and then applying a permutation of entries $\sigma\in \Aut(C)$ which maps the pair $(u,v)$ to $(0,e_1)$.
\end{proof}

Lemmas~\ref{rmelusive} and~\ref{reedmullercomptrans} complete the proof of Part 1 of Theorem~\ref{maintheorem}.

\section{\texorpdfstring{$s$}{s}-Elusive Codes}\label{selusive}

% \begin{lemma}\label{int4}
%  Suppose $\alpha,\beta\in \varGamma$ and $d(\alpha,\beta)=2s$, for some integer $s$. Then $\varGamma_s(\alpha)\cap\varGamma_s(\beta)= \{\gamma(\alpha|i,j|\beta_i,\beta_j)\mid i,j\in \diff(\alpha,\beta),i\neq j\}$.
% \end{lemma}

Let $C$ be a code in $H(m,q)$. Recall that $C$ is \emph{$s$-elusive} if $\Aut(C_s)$ is strictly larger than $\Aut(C)$. Note that for any $x\in \Aut(C_s)$ the code $C^x$ is equivalent to $C$, and thus has the same size and minimum distance, and has conjugate automorphism group.

\begin{lemma}\label{sneighbours}
 Let $C$ be an $s$-elusive code and $x\in \Aut(C_s)$. Then $(C_s)^x=(C^x)_s=C_s$.
\end{lemma}

\begin{proof}
 Note that $x\in\Aut(C_s)$ and thus fixes $C_s$ setwise, so it follows that $(C_s)^x=C_s$. It remains to be shown that $(C^x)_s=C_s$. Let $\nu\in C_s$ be distance $s$ from $\alpha\in C$. Then $d(\nu^x,\alpha^x)=s$. Suppose there exists some $\beta\in C^x$ such that $d(\nu,\beta)<s$. Then $d(\nu^{x^{-1}},\beta^{x^{-1}})<s$, however $\beta^{x^{-1}}\in C$, contradicting the fact that $x$ fixes $C_s$ setwise. Hence $\nu\in (C^x)_s$ and thus $(C^x)_s=C_s$, as these sets have the same size.
\end{proof}

If $C$ is an $s$-elusive code then there exists an automorphism $x\in \Aut(C_s)\setminus \Aut(C)$. This implies that $C^x\neq C$, so that there is some codeword $\alpha\in C$ such that $\alpha^x\notin C$.

\begin{definition}
 Let $C$ be an $s$-elusive code in $H(m,q)$, $x\in \Aut(C_s)\setminus \Aut(C)$ and $\alpha\in C$ such that $\alpha^x\notin C$. Then we call the triple $(C,\alpha,x)$ an \emph{$s$-elusive triple}.
\end{definition}

\begin{lemma}\label{unique}
 Let $(C,\alpha,x)$ be an $s$-elusive triple in $H(m,q)$ with $C$ having minimum distance $\delta\geq 2s+1$. Then, for all $\nu\in\varGamma_s(\alpha)$, there exists a unique $\pi\in C_{2s}\cap\varGamma_s(\nu)$ such that $\pi\in C^x$.
\end{lemma}

\begin{proof}
 Since $\delta\geq 2s+1$, the union $C_s=\cup_{\gamma\in C}\varGamma_s(\gamma)$ is disjoint. Now $C^x$ is equivalent to $C$ and, by Lemma~\ref{sneighbours}, $C_s^x=C_s$. Thus each $\nu\in C_s$ is distance $s$ from some vertex $\pi$ in $C^x$. That is, if $\nu\in\varGamma_s(\alpha)$ then there exists some vertex $\pi\in\varGamma_s(\nu)\cap C^x$. Now, $d(\alpha,\pi)\leq d(\alpha,\nu)+d(\nu,\pi)= 2s$ and hence $\pi\notin C$ since $\delta\geq 2s+1$. Moreover, this means $\pi\in C_k$, for some $k$ such that $1\leq k \leq 2s$.

 Suppose $\pi\in C_k$, where $1\leq k < 2s$. Then there exists $\beta\in C$ such that $\pi\in\varGamma_k(\beta)$, in particular there is a path of length $k$ from $\beta$ to $\pi$. Choose a vertex $\mu$ on this path, such that $\mu\in\varGamma_s(\beta)$. Then $\mu\in C_s$, however $d(\pi,\mu)=k-s<s$ contradicting the fact that $C_s^x=C_s$.

 Suppose there exists $\pi'\in \varGamma_s(\nu)\cap C^x$ such that $\pi'\neq \pi$. Then $\pi,\pi'$ are in the code $C^x$ which is equivalent to $C$. However $d(\pi,\pi')\leq d(\pi,\nu)+d(\nu,\pi')=2s$ contradicting $\delta=2s+1$. Thus $\pi$ is unique.
\end{proof}

The next definition introduces the concept of a $q$-ary $t$-design, which helps to describe the structure of an $s$-elusive code. Designs arise in many other contexts, for instance when considering $s$-regular codes \cite{delsarte1973algebraic}. First the notion of \emph{covering} a vertex is required.

\begin{definition}
 Let $0\in Q$ and $\nu,\alpha\in H(m,q)$. The vertex $\nu$ is said to be \emph{covered} by $\alpha$, if $\nu_i=\alpha_i$ for every $i\in M$ such that $\nu_i\neq 0$.
\end{definition}

In other words $\alpha$ \emph{covers} $\nu$ if each non-zero entry of $\nu$ agrees with the corresponding entry of $\alpha$.

\begin{definition}
 A \emph{$q$-ary $t$-$(m,k,\lambda)$ design} consists of a subset $\mathcal{D}\subseteq \varGamma_k(\b 0)$ of weight $k$ vertices of $H(m,q)$ such that each vertex $\nu \in\varGamma_t(\b 0)$ is covered by exactly $\lambda$ vertices of $\mathcal{D}$. When $q=2$, $\D$ is simply a $t$-$(m,k,\lambda)$ design and if additionally $\lambda=1$, $\mathcal{D}$ is called an $S(t,k,m)$ Steiner system.
\end{definition}

There are many examples where designs arise in coding theory. Theorem~\ref{design} should be compared, for example, with \cite[Theorem~2.12]{borges2019completely}, which states that the set of all weight $k$ vertices of a completely regular code having minimum distance $\delta$ in $H(m,q)$ form a $q$-ary $\lfloor \frac{\delta}{2}\rfloor$-$(m,k,\lambda_k)$ design, for some ineteger $\lambda_k$.

\begin{theorem}\label{design}
 Let $(C,\b 0,x)$ be an $s$-elusive triple in $H(m,q)$ with $\delta\geq 2s+1$. Then the set $\varGamma_{2s}(\b 0)\cap C^x$ forms a $q$-ary $s$-$(m,2s,1)$ design. In particular, if $q=2$, then $\varGamma_{2s}(\b 0)\cap C^x$ forms an $S(s,2s,m)$ Steiner system.
\end{theorem}

\begin{proof}
 By Lemma~\ref{unique}, every vertex of $\varGamma_s(\b 0)$ is covered by a unique element of $\varGamma_{2s}(\b 0)\cap C^x$, with respect to $\b 0$ and thus the result follows.
\end{proof}

This gives the following bound for the minimum distance of an $s$-elusive code.

\begin{theorem}\label{mindistupbound}
 Let $C$ be an $s$-elusive code in $H(m,q)$. Then
 \begin{enumerate}
   \item if $q=2$ then $\delta\leq 2s+2$, and,
   \item if $q\geq 3$ then $\delta\leq 2s+1$.
 \end{enumerate}
\end{theorem}

\begin{proof}
 If $\delta\leq 2s$, or $2s+1\geq m$, then the result holds trivially. Suppose $\delta\geq 2s+1$ and $2s+1<m$. Now, there exists some $x\in \Aut(C_s)$ and $\alpha\in C$ such that $\alpha^x\notin C$, where we may assume that $\alpha={\b 0}$. Then, by Theorem~\ref{design}, $\varGamma_{2s}({\b 0})\cap C^x$ forms a $q$-ary $s$-$(m,2s,1)$ design $\D$. Hence, for all $\mu\in\varGamma_s({\b 0})$, there exists some $\beta\in \varGamma_{2s}({\b 0})\cap C^x$ such that $\beta$ covers $\mu$.

 Suppose that $q=2$. Since $2s<m-1$, it follows that there exists some $i\in M$ such that $\beta_i=0$. Thus, there exists $\nu\varGamma_s({\b 0})$ with $\nu_i=1$ and $d(\mu,\nu)=2$. Note that $\beta$ does not cover $\nu$. Hence, there exists some block $\gamma$ of $\D$ covering $\nu$. It then follows from the triangle inequality that
 \[
  d(\beta,\gamma)\leq d(\beta,\mu)+d(\mu,\nu)+d(\nu,\gamma)= 2s+2.
 \]
 As $\beta,\gamma \in C^x$, and $C^x$ is equivalent to $C$, this proves part 1.

 Let $q\geq 3$. Choose $i\in M$ such that $\mu_i\neq 0$. Since $q\geq 3$, there exists an $a\in Q^\times$ such that $\mu_i\neq a$. Let $\nu\in\varGamma_s({\b 0})$ with $\nu_i=a$ and $\nu_j=\mu_j$ for $j\neq i$. Then $\beta$ does not cover $\nu$, so there exists a block $\gamma$ of $\D$ covering $\nu$. It then follows from the triangle inequality that
 \[
  d(\beta,\gamma)\leq d(\beta,\mu)+d(\mu,\nu)+d(\nu,\gamma)= 2s+1.
 \]
 Since $\beta,\gamma \in C^x$, and $C^x$ is equivalent to $C$, this proves part 2.
\end{proof}

The Preparata codes are a family of binary codes of length $2^{2d}$ for each integer $d\geq 2$. In addition to satisfying equations (4.1) and (4.2), codewords of the Preparata codes satisfy one extra non-linear equation. For a full definition see \cite[(16.12)]{cameron1991designs}, taking note that $\bar{\P} (\sigma)$ is denoted as $\P(2d)$ here, with $\sigma$ arbitrary.

\begin{proposition}\label{elusiveprep}
 The Preparata codes $\mathcal{P}(2d)$ are $2$-elusive codes.
\end{proposition}

\begin{proof}
 Let $C=\RM_2(2d,2d)$ and $\P=\P(2d)$. It suffices to prove that the $2$-neighbour sets $\P_2$ and $C_2$ are equal and that $\P$ is properly contained in $C$. It then follows that $\Aut(C)$ fixes $\P_2$ but not $\P$, since $\Aut(C)$ contains the translations by any codeword. Thus $\P$ is $2$-elusive.

 First, \cite[(16.12) (a) and (b)]{cameron1991designs} gives $\P\subset C$. Since $\delta(C)=4$ it follows that $\P_2\subseteq C_2$. Now, by Lemma~\ref{neighbourset}, $C$ has covering radius $2$ and dimension $2^{2d}-2d-1$. Hence $H(2^{2d},2)= C\cup C_1\cup C_2$. This gives
 \begin{align*}
  |C_2| &=\,|H(2^{2d},2)|-|C|-|C_1|\\
  &=\,2^{2^{2d}}-2^{2^{2d}-2d-1}-2^{2^{2d}-2d-1}\cdot 2^{2d}\\
  &=\,2^{2^{2d}-1}-2^{2^{2d}-2d-1}.
 \end{align*}
 Furthermore, by \cite[(16.16)]{cameron1991designs}, $\P$ has minimum distance $6$ so is properly contained in $C$. This also gives,
 \begin{align*}
  |\P_2| &=\,|\P|\binom{m}{2}(q-1)^2\\
  &=\,2^{2^{2d}-4d}2^{2d-1}(2^{2d}-1)\\
  &=\,2^{2^{2d}-1}-2^{2^{2d}-2d-1}.\qedhere
 \end{align*}
\end{proof}

\begin{corollary}\label{designprep}
 Let $\b 0\in \P(2d)$ and $x\in \Aut(C_2)\setminus \Aut(C)$. Then $\varGamma_4(\b 0)\cap \P(2d)^x$ is an $S(2,4,2^{2d})$ Steiner system.
\end{corollary}

\begin{proof}
 This follows from Theorem~\ref{design} and Lemma~\ref{elusiveprep}.
\end{proof}

There exists a $3$-$(22,6,1)$-design, namely the Witt design $W_{22}$. This suggests an elusive code with these parameters may exist. Indeed, taking the even weight subcode of the binary perfect Golay code $\mathcal{G}_{23}$ and puncturing the resulting code produces a $3$-elusive code.

\begin{proposition}\label{ewpgolaycode}
 Let $\mathcal{PG}$ and $\mathcal{EG}$ be the codes obtained by puncturing the binary perfect Golay code $\mathcal{G}_{23}$ and the even weight subcode of the Golay code $\mathcal{G}_{23}$, respectively. Then $\mathcal{PG}_3=\mathcal{EG}_3$ and $\mathcal{EG}$ is $3$-elusive with minimum distance $\delta=7$.
\end{proposition}

\begin{proof}
 Now $\mathcal{G}_{23}$ is a linear $[23,12,7]$ code with covering radius $3$, and $\PermAut(\G)^M\cong M_{23}$ is transitive on $M$. Thus, puncturing $\mathcal{G}_{23}$ results in the linear $[22,12,6]$ code $\PG$ with covering radius $\rho=3$. The even weight subcode of $\G$ is a linear $[23,11,8]$ code, again with $M_{23}$ acting as pure permutations on entries, so puncturing results in the $[22,11,7]$ code $\EG$.

 Since $\PG$ has covering radius $3$ and minimum distance $6$ it follows that $V\varGamma=\PG\cup \PG_1\cup \PG_2\cup \PG_3$, where this union is disjoint. So,
 \begin{align*}
  |\PG_3|= & |V\varGamma|-|\PG|-|\PG_1|-|\PG_2|\\
  = & 2^{22}-2^{12}-2^{12}\cdot 22-2^{12}\cdot \frac{22\cdot 21}{2}\\
  = & 2^{12}(2^{10}-1-22-11\cdot 21)\\
  = & 2^{13}\cdot 5\cdot 7\cdot 11.
 \end{align*}

 Now, $\EG$ has minimum distance $7$, so $|\EG_3|=2^{11}\cdot 22\cdot 21\cdot 20/6=2^{13}\cdot 5\cdot 7\cdot 11=|\PG_3|$. Since $\PG$ is linear, any translation by a vertex in $\PG\setminus\EG$ fixes $\PG_3=\EG_3$. However this automorphism is not an element of $\Aut(\EG)$.
\end{proof}

% 
% \bibliographystyle{plain}  % BibTeX instruction ... use style file: uwa.bst
% 			 % benign if BibTeX not used
% \bibliography{ref-dec18}       % BibTeX instruction ... read: uwa.bib
			 % Also tells LaTeX to read in the .bbl file
			 % (uwadissert.bbl in our case) ... if it exists.
\end{document}